\documentclass[10pt]{amsproc}

\usepackage{graphicx}

\usepackage{tikz}
\usetikzlibrary{calc}

\usepackage{amsmath}
\usepackage{amsthm}
\usepackage{amsfonts}
\usepackage{amssymb}
\usepackage{bigints}
\usepackage{mathtools}
\usepackage{microtype}

\DeclarePairedDelimiter\abs{\lvert}{\rvert}%
\copyrightinfo{}{}

\newtheorem{theorem}{Theorem}[section]
\newtheorem{lemma}[theorem]{Lemma}
\newtheorem{prop}[theorem]{Proposition}
\newtheorem{cor}{Corollary}

\theoremstyle{definition}

\theoremstyle{remark}
\newtheorem{remark}[theorem]{Remark}

\theoremstyle{definition}

\numberwithin{equation}{section}

\begin{document}

\title[Affine subspaces of curvatures from closed curves]{Affine subspaces of curvature functions from closed planar curves}


%    author one information

\author{Leonardo Alese}
\address{TU Graz, Department of Mathematics, Institute of Geometry }
\curraddr{}
\email{alese@tugraz.at}

%    \subjclass is required.

\date{\today}

%% IF YOU HAVE FONTS INSTALLED
%\usepackage{mtpro2}
%\usepackage{mathtime}

%\item times
%\item pifont
%\item graphicx
%\item color
%\item AMS styles: amsmath, amsthm, amsfonts, amssymb
%\item url

\begin{abstract}
Given a pair of real functions $(k,f)$, we study the conditions they must satisfy for $k+\lambda f$ to be the curvature in the arc-length of a closed planar curve for all real $\lambda$. Several equivalent conditions are pointed out, certain periodic behaviours are shown as essential and a family of such pairs is explicitely constructed. The discrete counterpart of the problem is also studied. Finally, the characterization obtained is used to show that a sufficient analogue of the 4-vertex theorem cannot be developed.
\end{abstract}

\maketitle

\section{Introduction.}

Closed curves are natural mathematical objects which have been studied since a long time. In his well-known paper \cite{fenchel} Fenchel makes the following comment about the study of the geometric properties of a space curve which depend on the assumption that the curve is closed.

\begin{quote}
\textit{The results are often comparatively elementary and seem to be isolated. On the other hand, the intuitive character of the statements and the lack of a general method of approach make the field attractive...}
\end{quote}
In this paper we focus on closed planar curves but the consideration above identifies pretty well the context of our contribution. In this area interesting questions keep on coming up, as in the recent \cite{alese2020}, where a surprising result on permuting arcs of a $C^1$ curve with the goal to make the curve closed is proven with elementary topological tools.

The natural and complete geometric descriptor we associate to a curve is its \textit{curvature}. If $\gamma \in C^2([0,2\pi],\mathbb{R}^2)$ is an \textit{arc-length parametrized planar curve}, \textit{i.e.}, a twice differentiable function from the interval $I:=[0,2\pi]$ to the real plane such that the norm of its first derivative $\| \gamma'  \|$ is constantly equal to $1$, we can define a \textit{turning angle} function $\theta$ that satisfies $\gamma'(t)=(\cos \theta(t), \sin \theta(t))=\mathrm{e}^{i\theta(t)}$, where $\mathbb{R}^2$ has been identified with the complex plane $\mathbb{C}$. The \textit{curvature} $k$ of $\gamma$ is defined as the first derivative $\theta'$ of the turning angle function. The other way round, given a continuous curvature $k\in C^0(I, \mathbb{R})$ we can reconstruct by integration, uniquely up to rigid motions, the curve it comes from. In fact, $\theta(t)=\int_0^t k(s)\mathrm{d}s + C$ and $\gamma(t)=\int_0^t \mathrm{e}^{i\theta(s)}\mathrm{d}s + V$. For a more extensive treatment of the subject the reader may refer to \cite{docarmo}.

Given another function $f\in C^0(I,\mathbb{R})$, the main question we are interested in this paper is: \begin{quote}
What are the conditions on $k$ and $f$ for $k+\lambda f$ to be the curvature of a closed curve for all $\lambda \in \mathbb{R}$? 
\end{quote} 
Here and in the following with \textit{closed} we just mean that starting and end point of the curve coincide (we will see though that the nature of the problem entails much stiffer relations also on the derivatives at the extreme points of the curve). Figure \ref{imgintro} visualizes the objects we are going to study. 

Interpolation of curvature functions is a tool used in computer graphics to gradually transform one curve into another, while mantaining the length of the curve \cite{metam}. This method does not perform well when it comes to deform closed curves, since there is no guarantee that the intermediate curves are closed as well; from the point of view of computer graphics this problem can be fixed by approximating the transition curves with closed ones that are not too far away from them \cite{curvbased}. In this paper we approach the problem from the theoretical perspective, exploring the conditions guaranteeing that all the curves are closed over the interpolation of the curvature functions. 

In the more general framework of deformations, the evolution of curves under the action of different flows has been studied: in \cite{flows} a curvature-based flow is used to transform a shape into another while preserving the length; the closedness of the curve over the process is guaranteed by an extra projection step to the hyperspace of $L^2$ defined by the constraint $\int_0^{2\pi} k'(s) \gamma(s) \mathrm{d}s = 0$ relating the position of the curve and the first derivative of its curvature, which is interestingly proven as a necessary and sufficient condition for a curve to be closed.

As for an outline of the contents, \S \ref{first} presents the main characterization theorem, proving also that the existence of a single affine line of curvature functions from closed curves is equivalent to the existence of an infinite dimensional affine space of such functions. In \S \ref{per} periodicity properties of $k$ and $f$ are shown. On the existence side, \S \ref{secexistence} deals with the explicit construction of pairs of analytic function $(k,f)$ that satisfy our constraints. In \S \ref{discrete} we discuss the discrete case. Finally, in \S \ref{nonex} a \textit{hardness} result on the task of telling whether a curve is closed by looking at its curvature $k$ is obtained. We show that it is not possible to develop a procedure that tells whether the associated curve is closed or not by accessing finitely many evaluations and/or level sets of $k$, its derivatives and its antiderivatives. 

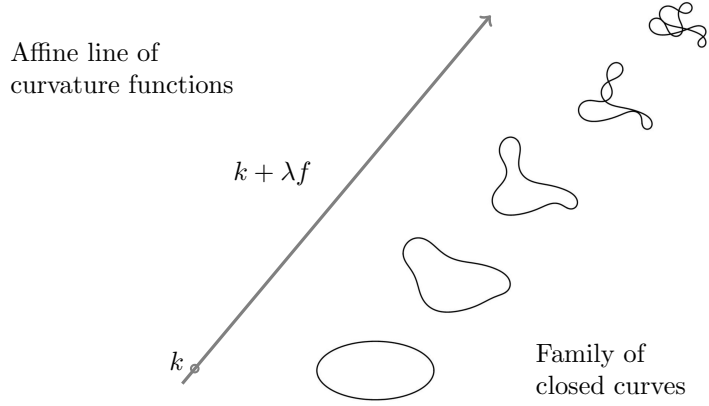
\begin{figure}
\begin{tikzpicture}
    \node[anchor=south west,inner sep=0] at (3,0) {\scalebox{0.42}{\includegraphics[width=\textwidth]{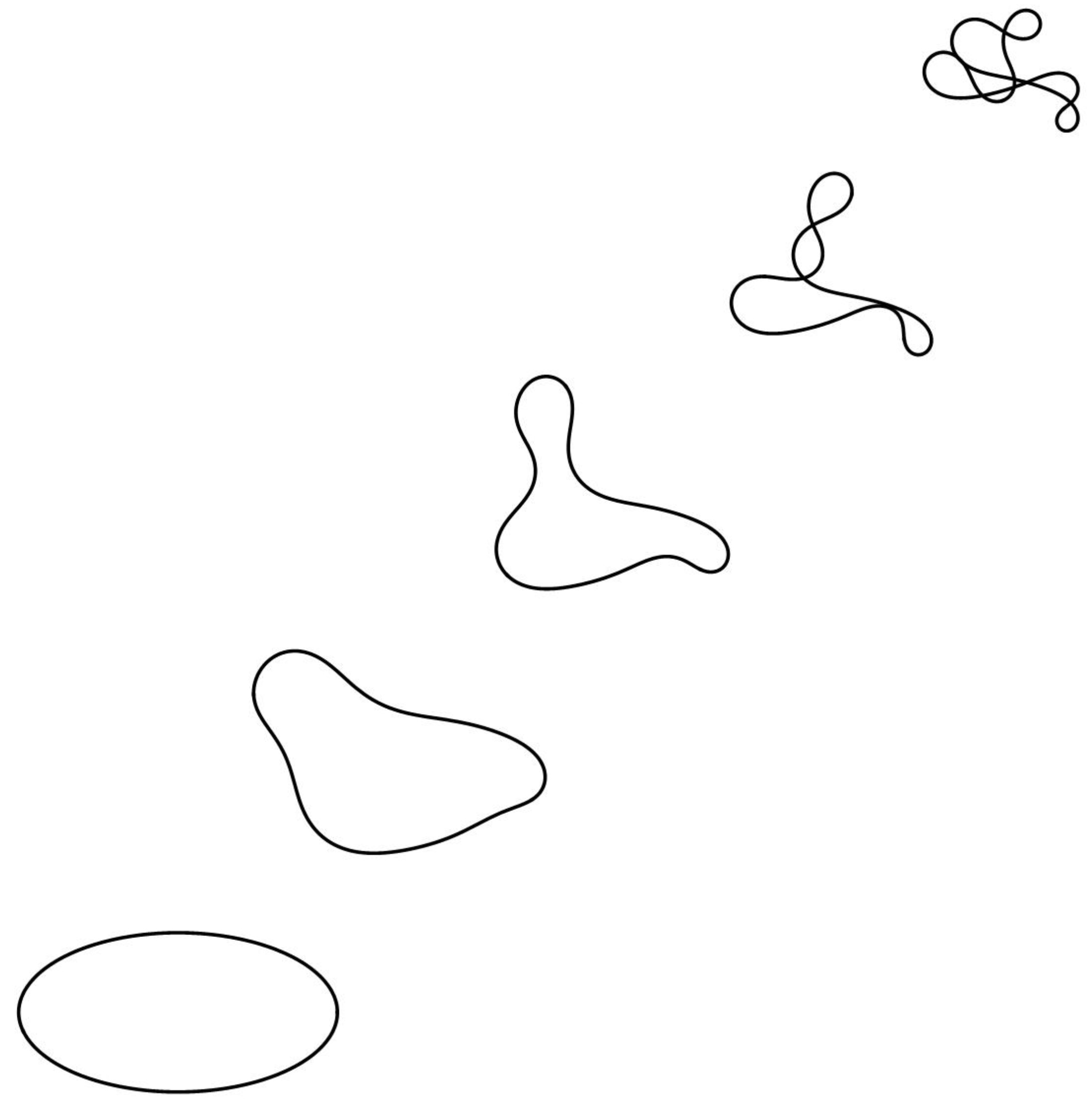}}};
    \draw[gray,very thick,->] (1.3,0.3) -- (5.4,5.2);
    \node[align=left] at (0.5,4.5) {Affine line of\\ curvature functions};
    \node[align=left] at (7,0.5) {Family of\\ closed curves};
    \draw[gray, thick] (1.4654,0.5) circle (0.05cm);
    \node[align=left] at (1.227,0.6) {$k$};
    \node[align=left] at (2.5,3.1) {$k+\lambda f$};
\end{tikzpicture}
\caption{An ellipse is deformed by adding multiple of $f$ to its curvature $k$. If, as in this case, $f$ is chosen properly, then all the curves of the family are closed. We want to study the constraints $k$ and $f$ must satisfy to present such a behaviour.} \label{imgintro}
\end{figure}

%If $\gamma$ is a $C^n$ arc-length parametrized curve and $l\leq n$, we say $\gamma$ is $C^l$-\textit{closed} if $\forall \; 0 \leq h \leq l$ it holds $\gamma^{(h)}(0) = \gamma^{(h)}(2\pi)$, that is the first $l$ derivatives coincide at the starting and end point. Since by identifying the extreme points of $I$ we obtain the circle $S^1$, this is equivalent to saying that $\gamma \in C^l(S^1,\mathbb{R}^2)$. When a curve is $C^0$-closed we just say it is \textit{closed} and when it is $C^l$-closed for all $l$ then we say it is $C^\infty$-closed.

\section{Equivalent characterizations of closedness.} \label{first}
Let $\gamma$ be a closed $C^2$ curve defined on the interval $I=[0,2\pi]$, $\theta$ its associated turning angle function and $k=\theta'$ its curvature. For $f\in C^0(I, \mathbb{R})$, we want to answer the question: what are the conditions on $f$ for $k+\lambda f$ to be the curvature of a closed curve for all $\lambda \in \mathbb{R}$? Calling $\phi(t):=\int_0^t f(s)\mathrm{d}s$, this is equivalent to
$$
\int_0^{2\pi} \mathrm{e}^{i(\theta(t)+\lambda \phi(t))} \mathrm{d}t=0, \;\;\;\;\;\; \forall \; \lambda \in \mathbb{R}.
$$
The function $F(\lambda):=\int_0^{2\pi} \mathrm{e}^{i(\theta(t)+\lambda \phi(t))}\mathrm{d}t$ is analytic in $\lambda$. This can be seen for example by giving the following explicit entire expansion for the real part of $F$ (the imaginary part is analogous):
$$
F_1(\lambda)=\sum_{c=0}^{\infty} c_n \frac{\lambda^n}{n!}, \;\; \text{with } c_n =
\begin{cases}
(-1)^{\frac{n}{2}} \;\;\; \int_0^{2\pi} \phi(t)^n \cos\theta(t)
\mathrm{d}t, & \;\; \text{if } n \text{ is even}, \\
(-1)^{\frac{n+1}{2}} \int_0^{2\pi} \phi(t)^n \sin\theta(t)\mathrm{d}t, & \;\; \text{if } n \text{ is odd}.
\end{cases}
$$
This observation alone is enough to conclude the first of our equivalent conditions.
\begin{lemma}
Let $k,f \in  C^0(I,\mathbb{R})$. Then the curve with curvature $k+\lambda f$ is closed $\forall \; \lambda \in \mathbb{R}$ $\Leftrightarrow$ we have the equality
\begin{equation} \label{cond1}
\int_0^{2\pi} \mathrm{e} ^{i \theta (t)}   \phi (t)^n \mathrm{d}t=0,  \;\;\; \forall \; n \in \mathbb{N}_0, 
\end{equation}
where $\theta(t)=\int_0^t k(s)\mathrm{d}s$ and $\phi(t) = \int_0^t f(s) \mathrm{d}s$. 
\end{lemma}
\begin{proof}
An analytic function is everywhere $0$ if and only if all of its derivatives vanish in at least one point. We conclude by computing the $n$-th derivative of $F$ and evaluating it in $\lambda=0$, obtaining $0=F^{(n)}(0)= i^n \int_0^{2\pi} \mathrm{e} ^{i \theta (t)}   \phi (t)^n \mathrm{d}t$. Note that we could take the derivative within the integral thanks to the Leibniz integral rule.
\end{proof}

We want now to better understand this condition, by discussing some of its implications. Our main tool will be an approximation argument based on the observation that, if $\phi$ satisfies the condition above, then for any $N \in \mathbb{N}_0$ and $(c_j)_{j\in\{1,...,N\}} \in \mathbb{R}^{N+1}$, also the sum $\sum_{j=0}^N c_j \phi^j$ does.

\begin{lemma} \label{composition}
$\theta,\phi \in C^1(I,\mathbb{R})$ satisfy condition (\ref{cond1}) $\Leftrightarrow$ $\theta,g(\phi)$ (composition of functions) satisfy condition (\ref{cond1}), for any $g$ bounded and integrable.
\end{lemma}

\begin{proof}
The `if' part is trivial. For the `only if' we use a density property of polynomials in our class of functions to approximate $g$. More explicitely, for any $n \in \mathbb{N}_0$ and $\varepsilon > 0$, there exists a polynomial $p_{n,\varepsilon}$ of degree $N(n,\varepsilon)$ such that 
$$
\int_0^{2\pi}  \big| g(\phi(t))^n - p_{n,\varepsilon}(\phi(t)) \big| \mathrm{d}t  <\varepsilon, 
$$ 
which implies
\begin{align*}
&\abs*{\int_0^{2\pi} \mathrm{e}^{i \theta(t)} g(\phi(t))^n \mathrm{d}t }  \\
\leq &  \abs*{\int_0^{2\pi} \mathrm{e}^{i \theta(t)} \bigl( g (\phi(t))^n - p_{n,\varepsilon}(\phi(t))\bigl) \mathrm{d}t}+  \abs*{     \int_0^{2\pi} \mathrm{e}^{i \theta(t)} p_{n,\varepsilon}(\phi(t))  \mathrm{d}t  } \\
 \leq & \int_0^{2\pi}  \big| g (\phi(t))^n - p_{n,\varepsilon}(\phi(t)) \big| \mathrm{d}t \leq \varepsilon. \qedhere
\end{align*}
\end{proof}

By Lemma \ref{composition}, the existence of $\phi$ satisfying (\ref{cond1}) implies the existence of an infinite-dimensional affine space through $\theta$ whose elements satisfy (\ref{cond1}) as well. From the perspective of the curvature, what we are saying here is that, choosing $g$ to be $C^1$, we can pass from $f$ to $fg'(\phi)$ and still have that the curves with curvature functions $k+\lambda fg'(\phi)$ are closed for all $\lambda$.

Before moving to the next lemma, which provides a much more local characterization of our constraint, it is convenient to recall that a level set $\phi^{-1}(a)=\{t \mid \phi(t)=a \}$ consists of isolated points, if $\phi'(t)\neq 0$ for all $t \in \phi^{-1}(a)$. For $\phi$ defined on a compact interval $I$, level sets of such regular values are therefore finite.

%\begin{remark} \label{obs1}
%If $\phi \in C^1(I,\mathbb{R})$ and $a$ is not a critical value of $\phi$, \textit{i.e. }$0 \notin %\phi'(\phi^{-1}(a))$, then $\phi^{-1}(a)$ is finite. In fact, if we assume $\phi^{-1}(a)$ is not finite% then there exists a strictly monotone sequence $\{b_n\}_{n \in \mathbb{N}} \subseteq \phi^{-1}(a)$ that% converges to $\bar{b}\in\phi^{-1}(a)$, by compactness and continuity. But then
%$$
%\phi'\big(\bar{b}\big)= \lim_{n \to \infty} \frac{\phi(b_n)-\phi(b)}{|b_n-b|}=0,
%$$
%which is a contradiciton.
%\end{remark}

\begin{lemma}
If $\theta,\phi \in C^1(I,\mathbb{R})$ satisfy condition (\ref{cond1}), then we have the implication

\begin{align} \label{cond3}
a \neq \phi(0),\phi(2\pi) \text{ is a regular value of } \phi \Rightarrow \sum_{b\in \phi^{-1}(a)} \frac{\mathrm{e}^{i\theta(b)}}{|\phi'(b)|}=0.    
\end{align}

\end{lemma}
\begin{proof}
In Lemma \ref{composition} we select $g=\chi_{[a,a+\delta]}$, that is the characteristic function of the interval $[a,a+\delta]$, and obtain 
$$
\int_{\phi^{-1}([a,a+\delta])} \mathrm{e}^{i \theta(t)} \phi(t) \mathrm{d}t =0,  \;\;\; \forall \; \delta\geq 0, a \in \mathbb{R}.
$$
We can find $\delta>0$ such that the restrictions $\{\phi_j\}$ of $\phi$ to the finitely many components of $\phi^{-1}([a,a+\delta])$ are invertible. Calling $R(\delta):=\int_{\phi^{-1}([a,a+\delta])} \mathrm{e}^{i \theta(t)} \phi(t) \mathrm{d}t$, we compute its derivative with respect to $\delta$
$$
R'(\delta) = \sum_j \mathrm{e}^{\theta(\phi_j^{-1}(a+\delta))} \phi(\phi_j^{-1}(a+\delta)) \cdot |(\phi^{-1})'(a+\delta)|,
$$
and then, since $R$ is a constant,
\begin{equation*}
0=R'(0)=a \sum_{b\in \phi^{-1}(a)} \frac{\mathrm{e}^{i\theta(b)}}{|\phi'(b)|}. \qedhere
\end{equation*}
\end{proof}
The reason we excluded the level sets $\phi(0)$ and $\phi(2\pi)$ from the constraint on the sum is to avoid to distinguish cases depending on the sign of the derivative at extreme points of the interval: all relevant information is already in the condition for the sum over inner points.

\begin{remark}
In order to have the equivalence (\ref{cond1}) $\Leftrightarrow$ (\ref{cond3}), in addition we must require $\int_{\phi^{-1}(a)}\mathrm{e}^{i\theta}=0$, for all $a\in \mathbb{R}$. If $\phi$ is analytic this requirement is always met.
\end{remark}

We collect in one theorem all the conditions we have proven equivalent.
\begin{theorem} \label{thmcond}
Let $k,f \in  C^0(I,\mathbb{R})$ and $\theta(t)=\int_0^t k(s)\mathrm{d}s, \phi(t) = \int_0^t f(s) \mathrm{d}s$. The following conditions are equivalent. 
\begin{enumerate}
\setcounter{enumi}{-1}
\item The curve with curvature $k+\lambda f$ is closed $\forall \; \lambda \in \mathbb{R}$,
\item $\int_0^{2\pi} \mathrm{e} ^{i \theta (t)}   \phi (t)^n \mathrm{d}t=0,  \;\;\; \forall \; n \in \mathbb{N}_0, $
\item $\int_0^{2\pi} \mathrm{e} ^{i \theta (t)} g( \phi (t) )^n \mathrm{d}t=0, \;\;\; \forall \; n\in \mathbb{N}_0 \text{ and any } g \text{ bounded and integrable}$.
\end{enumerate}
Moreover, they imply
$$
a \neq \phi(0),\phi(2\pi)  \text{ is a regular value of } \phi \Rightarrow \sum_{b\in \phi^{-1}(a)} \frac{\mathrm{e}^{i\theta(b)}}{|\phi'(b)|}=0.   
$$
\end{theorem}

We also point out the following corollary, which rules out the possibility of vector spaces of curvatures of closed curves.
\begin{cor}
For $f \in  C^0(I,\mathbb{R})$, there exists $\lambda$ such that the curve that has $\lambda f$ as curvature is not closed. More precisely, the set  $\Lambda=\{ \lambda \in \mathbb{R} : \lambda f $ is the curvature of a closed curve$\}$ does not have accumulation points.
\end{cor}

\begin{proof}
Setting $k\equiv 0$, condition (1) of Theorem \ref{thmcond} becomes $\int_0^{2\pi} \phi(t)^n dt=0$ for all $n\in\mathbb{N}_0$, which for $n$ even can only be satisfied by $\phi \equiv 0$. On the other hand, the presence of accumulation points in $\Lambda$ is enough to guarantee $\Lambda=\mathbb{R}$ by the analiticity argument from the beginning of the section, hence entailing the same conclusion.
\end{proof}

\section{Conditions on the boundary.} \label{per} In this section we discuss some periodicity properties that $\theta$ and $\phi$ must satisfy if the curve with turning angle function $\theta+\lambda \phi$ is closed for all $\lambda\in \mathbb{R}$. We will show that, under the conditions of Theorem \ref{thmcond}, the respective behaviour of $\theta$ and $\phi$ on the boundary is strongly related.

\begin{prop} \label{bound1}
If $\theta,\phi \in C^h(I,\mathbb{R})$ satisfy condition (\ref{cond3}) and $\phi(0)$ is not a critical value, then $\phi(0)=\phi(2\pi)$ and the derivatives of $\phi,\theta$ obey either
\begin{align*}
\theta(2\pi)-\theta(0)\equiv0 \mod 2\pi, \;\; &\theta^{(k)}(0)= \theta^{(k)}(2\pi), \; &\;\;\;\;\;\;\;\;\;1\leq k \leq h-1, \\
&\phi^{(k)}(0)= \phi^{(k)}(2\pi), &\;\;\;\;\;\;\;\;\; 1\leq k \leq h
\end{align*}
or
\begin{align*}
\theta(2\pi)-\theta(0)\equiv\pi \mod 2\pi, \;\; &\theta^{(k)}(0)= (-1)^k \theta^{(k)}(2\pi), \; &1\leq k \leq h-1, \\
&\phi^{(k)}(0)= (-1)^k \phi^{(k)}(2\pi), \; &1  \leq k \leq h.
\end{align*}
\end{prop}
\begin{proof}
Setting $a=\phi(0)$, we consider $\delta>0$ such that $\phi$ is invertible on the finitely many components of $\phi^{-1}([a-\delta,a+\delta])$. We then look at the connected components of $\phi^{-1}([a,a+\delta])$ and we use the symbol $p_j$ for the restriction of $\phi$ to the $j$-th component, numbered from the left ($j=1,...,N_+$, see Figure \ref{phi}). Similarly, functions $m_j$'s are the restriction of $\phi$ to the connected components of $\phi([a-\delta,a])$. Rewriting condition (\ref{cond3}) we have
$$
\lim_{\varepsilon \to 0^+} \sum_j \frac{\mathrm{e}^{i\theta\big(p_j^{-1}(a+\varepsilon)\big )}}{|\phi'(p_j^{-1}(a+\varepsilon))|} = \lim_{\varepsilon \to 0^-} \sum_j \frac{\mathrm{e}^{i\theta\big(m_j^{-1}(a+\varepsilon)\big )}}{|\phi'(m_j^{-1}(a+\varepsilon))|}=0.
$$

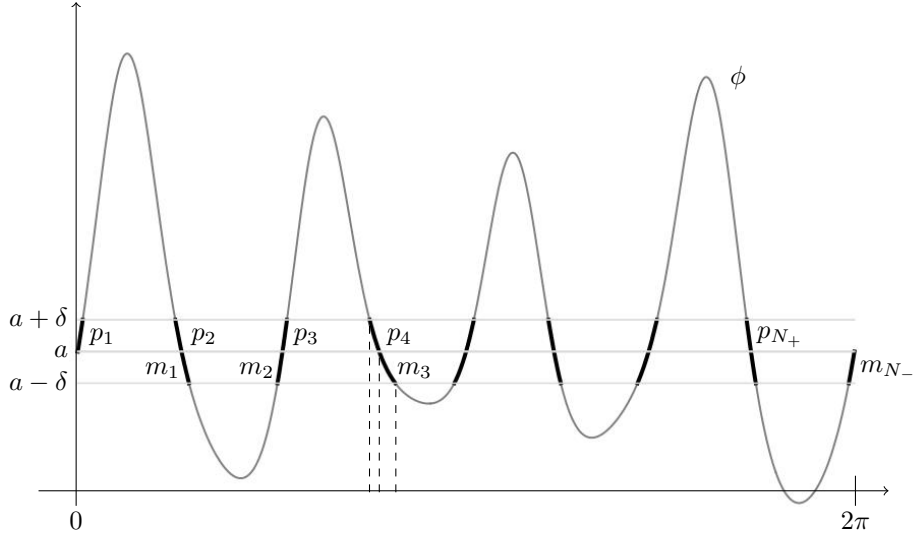
\begin{figure}
\begin{tikzpicture}
    \node[anchor=south west,inner sep=0] at (0,0) {\scalebox{0.85}{\includegraphics[width=\textwidth]{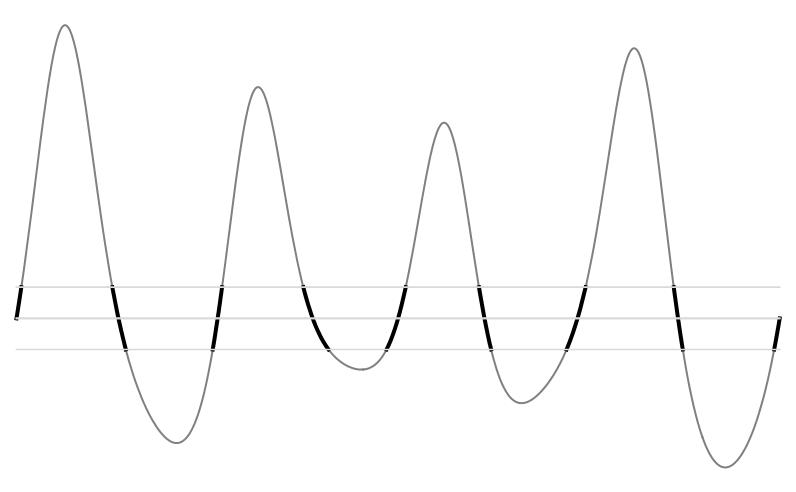}}};
     \draw[->] (0.2,0.3) -- (0.2,7);
       \draw[->] (-0.3,0.5) -- (11,0.5);
       \draw[] (10.56,0.3) -- (10.56,0.7);
           \node[] at (9,6) {$\phi$};
       \node[] at (10.56,0.1) {$2\pi$};
       \node[] at (0.2,0.1) {$0$};
    \node[] at (0,2.35) {$a$};
    \node[] at (-0.3,2.8) {$a+\delta$};
    \node[] at (-0.3,1.95) {$a-\delta$};
    \node[] at (0.55,2.55) {$p_1$};
    \node[] at (1.9,2.55) {$p_2$};
    \node[] at (3.26,2.55) {$p_3$};
    \node[] at (4.5,2.55) {$p_4$};
    \node[] at (9.56,2.55) {$p_{N_+}$};
    \node[] at (1.35,2.1) {$m_1$};
    \node[] at (2.6,2.1) {$m_2$};
    \node[] at (4.7,2.1) {$m_3$};
    \node[] at (11,2.1) {$m_{N_-}$};
    \draw[dashed] (4.23,0.48) -- (4.23,2.4);
    \draw[dashed] (4.1,0.48) -- (4.1,2.76);
    \draw[dashed] (4.45,0.48) -- (4.45,1.92);
\end{tikzpicture}
\caption{The functions $\{p_j\}_{j \in \{1,...,N_+\}}$ are the invertible restrictions of $\phi$ to the finitely many components of $\phi^{-1}([a,a+\delta])$ and analogously $\{m_j\}_{j \in \{1,...,N_-\}}$ are the restrictions of $\phi$ to $\phi^{-1}([a-\delta,a])$. } \label{phi}
\end{figure}

Since limit contributions coming from the restrictions to intervals in the interior of $I$ are equal in the two sums, we have no other choice than $\phi(0)=\phi(2\pi)$, otherwise the contribution from $p_1$ in the first sum could not be balanced in the limit by any terms of the second sum. Without loss of generality we can assume $\phi'(0)>0$. We distinguish two cases, depending on the sign of $\phi'(2\pi)$. If $\phi'(2\pi)>0$, just by rewriting again condition (\ref{cond3}) while keeping contributions from the two extreme intervals on the left-hand side of the equalities, we have , for $0< \varepsilon <\delta$
\begin{align*}
\frac{\mathrm{e}^{i\theta\big(p_1^{-1}(a+\varepsilon)\big )}}{\phi'(p_1^{-1}(a+\varepsilon))}&=- \;  \sum_{1<j} \frac{\mathrm{e}^{i\theta\big(p_j^{-1}(a+\varepsilon)\big )}}{|\phi'(p_j^{-1}(a+\varepsilon))|}, \\  
\frac{\mathrm{e}^{i\theta\big(m_{N_-}^{-1}(a-\varepsilon)\big )}}{\phi'(m_{N_-}^{-1}(a-\varepsilon))}&=-\sum_{j<N_-}
\frac{\mathrm{e}^{i\theta\big(m_j^{-1}(a-\varepsilon)\big )}}{|\phi'(m_j^{-1}(a-\varepsilon))|}.
\end{align*}
The sums on the right-hand side of the equations are equal for $\varepsilon = 0$, entailing 
$$
\frac{\mathrm{e}^{i\theta(0)}}{\phi'(0)}=\frac{\mathrm{e}^{i\theta(2\pi)}}{\phi'(2\pi)},
$$
which proves $\theta(2\pi)-\theta(0)\equiv 0 \mod 2\pi$ and $\phi'(0)=\phi'(2\pi)$. Analogously, taking the first derivative of the equations with respect to $\varepsilon$ and considering the limit $\varepsilon \to 0$, we conclude
$$
\frac{i\mathrm{e}^{i\theta(0)}\theta'(0)-\mathrm{e}^{i\theta(0)}\phi''(0)\frac{1}{\phi'(0)}}{\phi(0)^2}=\frac{i\mathrm{e}^{i\theta(2\pi)}\theta'(2\pi)-\mathrm{e}^{i\theta(2\pi)}\phi''(2\pi)\frac{1}{\phi'(2\pi)}}{\phi(2\pi)^2},
$$
which, already knowing the respective relations of $\theta, \phi$ and $\phi'$ at extreme parameters, and noticing that $\mathrm{e}^{i\theta(0)}$ and $i \mathrm{e}^{i\theta(0)}$ are orthogonal, implies $\theta'(0)=\theta'(2\pi)$ and $\phi''(0)=\phi''(2\pi)$. For the derivatives of higher order, the statement follows analogously by  induction.
 
If $\phi'(2\pi)<0$, we get
\begin{align*}
\frac{\mathrm{e}^{i\theta\big(p_1^{-1}(a+\varepsilon)\big )}}{\phi'(p_1^{-1}(a+\varepsilon))}
-\frac{\mathrm{e}^{i\theta\big(p_{N_+}^{-1}(a+\varepsilon)\big )}}{\phi'(p_{N_+}^{-1}(a+\varepsilon))}&=-\sum_{1<j<N_+} \frac{\mathrm{e}^{i\theta\big(p_j^{-1}(a+\varepsilon)\big )}}{|\phi'(p_j^{-1}(a+\varepsilon))|}, \\
0 &= - \;\;\; \sum_{j} \;\;\,
\frac{\mathrm{e}^{i\theta\big(m_j^{-1}(a-\varepsilon)\big )}}{|\phi'(m_j^{-1}(a-\varepsilon))|} 
\end{align*}
and we conclude again by taking derivatives term by term with respect to $\varepsilon$ and using induction.
\end{proof}

\begin{remark} 
Note that the constraint on the curves associated to $k+\lambda f$ to be closed for all $\lambda$ just means that starting and end point coincide. Proposition \ref{bound1} proves that in this case the function $k$ and $f$ enjoy much stronger periodicity.
\end{remark}

\begin{remark} 
In the hypotheses of Proposition \ref{bound1}, we obtain an additional constraint on the integral over $I$ of the function $f$, in fact $\int_0^{2\pi} f(s) \mathrm{d}s = \phi(2\pi) =\phi(0)=0$. This means that along the affine line $k+\lambda f$ the total turning angle of the associated curve is constant and equal to $0$ or $\pi$ up to multiples of $2\pi$.
\end{remark}

\section{Explicit constructions of families of closed curves.} \label{secexistence} In \S \ref{first} and \S \ref{per} we characterized pairs of functions $(k,f)$ such that the curve obtained by integrating the curvature $k+\lambda f$ is closed for all $\lambda \in \mathbb{R}$. In this section we are interested in the existence of such pairs. We show how one can explicitly construct curvature functions with the desired properties.  

\begin{lemma} \label{existence}
If $\theta \in C^1(I,\mathbb{R})$, then 
\begin{align*}
\exists \phi \in C^1(I,\mathbb{R}): \int_0^{2\pi} & \mathrm{e}^{i\theta(s)} \phi(s)^n \mathrm{d} s=0, \; \forall n\in \mathbb{N} \\
&\Leftrightarrow \\ 
\exists \psi \in C^1(I,\mathbb{R}): \int_0^{2\pi} & \mathrm{e}^{i\theta(s)} \mathrm{e}^{i n \psi(s)} \mathrm{d}s=0, \; \forall n\in \mathbb{N}. \\
\end{align*}
\end{lemma}
\begin{proof}
The first existence statement implies the second just by taking $\psi=\phi$ and recalling condition (2) of Theorem \ref{thmcond}, guaranteeing that the composition with a function that is bounded and integrable mantains the desired property. The other way round we pick for example $\phi=\cos(\psi)$ and conclude by observing that $\cos(\psi)^n$ can be rewritten as a linear combination of terms of the form $\cos(h\cdot\psi)$.
\end{proof}

We now consider curves allowing a periodic regular parametrization that can be expressed as a Fourier series with periodic gaps in the coefficients 
$$
\gamma(t)=\Bigg(\sum_{j=0}^\infty a_j \cos(j\cdot t)+b_j \sin(j\cdot t),\sum_{j=0}^\infty \bar{a}_j \cos(j\cdot t)+\bar{b}_j \sin(j\cdot t)\Bigg),
$$
that is $a_j=b_j=\bar{a}_j=\bar{b}_j=0$ whenever $j$ is an integer multiple of $M\in\mathbb{N}$.
The asymptotics of the coefficients for $j$ going to infinity determines periodicity and differentiability of the function (see for example \cite{harmonic}). From now on we assume that $\gamma$ is a closed analytic curve, which, in the most trivial case, can simply be obtained by truncating the series and considering a trigonometric polynomial; in this case all the harmonics with index larger than the degree of the polynomial are $0$ and therefore there exists always $M$ satisfying the conditions above. By the orthogonality relations between elements of a Fourier basis we have
$$
\int_0^{2\pi} \gamma'(t) \cos(n\cdot M\cdot t) dt =(0,0), \; \forall n \in \mathbb{N}.
$$
Using complex notation, we write $\gamma'(t)$ as $ v(t)\mathrm{e}^{i\theta(t)}$ where $v(t)=\|\gamma'(t)\|$ is the speed of $\gamma$ and $\theta$ is the turning angle associated to the parametrization. After reparametrizing with respect to the arc-length (always possible as long as the curve is regular) we obtain, possibly scaling our curve to a length of $2\pi$, 
$$
\int_0^{2\pi} \mathrm{e}^{i\theta(t(s))} \cos(n\cdot M\cdot t(s)) ds =0, \; \forall n \in \mathbb{N}.
$$
Note that if $\gamma$ is analytic than also $\gamma'$, $\|\gamma'\|$, $\int_0^t \|\gamma'\|$ and its inverse are analytic and therefore arc-length parametrization preserves analyticity.
By Lemma \ref{existence}, $\phi(s)=\cos(l\cdot t(s))$ and $\theta(t(s))$ satisfy condition (1) of Theorem \ref{thmcond} and therefore the analytic curve obtained by integrating $\mathrm{e}^{i(\theta+\lambda \phi)}$ is closed for all real $\lambda$'s (note that the functions $\phi$ and $\theta$ constructed this way are in general not periodic of any period smaller than $2\pi$). It is enough to take the derivative with respect to $s$ to get the correspondent curvature functions. Figure \ref{imgintro} and Figure \ref{family} show families of curves obtained by such a linear modification of the turning angle (or equivalently of the curvature).

\begin{figure}
\begin{tikzpicture}
    \node[anchor=south west,inner sep=0] at (0,0) {\scalebox{1}{\includegraphics[width=\textwidth]{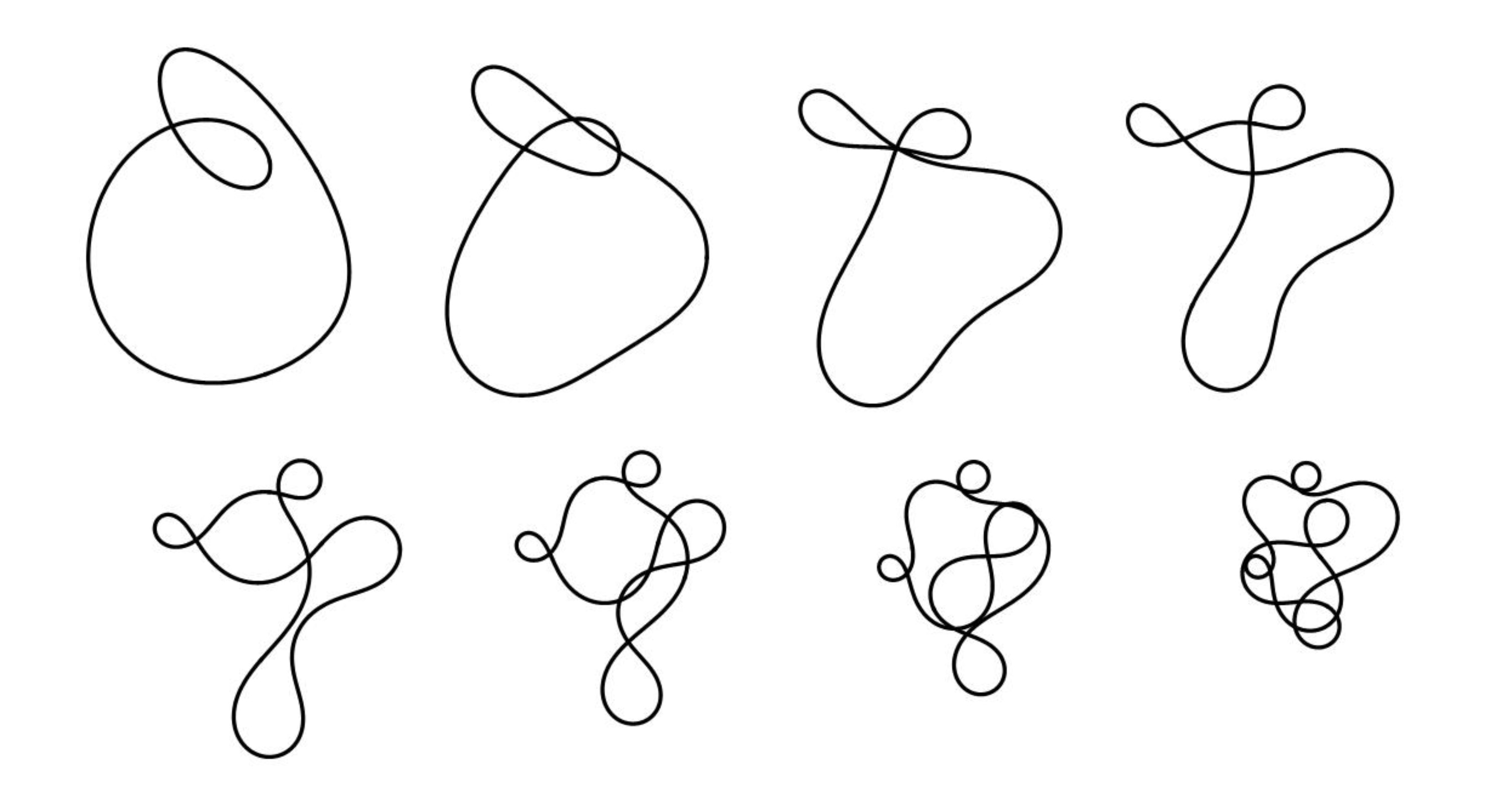}}};
\end{tikzpicture}
\caption{The turning angle $\theta$ of a trigonometric curve of degree $3$ is linearly changed to $\theta+\lambda\phi$ with $\phi(t)=\mathrm{e}^{\cos(4t)}+2\cos(4t)$ while the curve remains closed. From top left to bottom right $\lambda$ goes from $0$ to $0.7$ by $0.1$ increments.} \label{family}
\end{figure}

\section{The discrete case.} \label{discrete} In this section we look at a discretization of the problem we studied in the smooth setting. Consider an arc-length parametrized polyline, that is a finite sequence of \textit{vertices} $(v_j)_{j \in \{1,2,...,N\}} \subset \mathbb{C}$ with $\|v_{j+1} - v_{j}\|=1$ for $1\leq j \leq N-1$. We define the \textit{curvature} $k_j$ at a non-extreme vertex $v_j$ as the counter-clockwise angle between $v_{j} - v_{j-1}$ and $v_{j+1} - v_{j}$. The \textit{turning angle} $\theta_j$ at an interior vertex $v_j$ is the sum $\sum_{r=2}^j k_j$. Also in this setting we can reconstruct, up to rigid motions, a polyline from its curvature, first computing the turning angle $(\theta_j)$ and then defining 
$$
v_1=0, \;\; v_2=1, \;\; v_j=v_{j-1}+\mathrm{e}^{i\theta_{j-1}} \text{ for } j\geq 3.
$$ 

We consider now a polyline with $N$ vertices, which is \textit{closed} ($v_1=v_N$) and whose curvature is $(k_j)$. Given another discrete function $(f_j)_{j \in \{2,...,N-1\}} \in \mathbb{R}^{N-2}$, we ask what are the conditions on $(f_j)$ to guarantee that the polyline with curvature $(k_j)+\lambda (f_j)$ is closed for all $\lambda \in \mathbb{R}$. The following theorem answers this question, drawing a strong analogy to Theorem \ref{thmcond}.

\begin{theorem} \label{discretetheorem}
Let $(k_j)$ and $(f_j)$ be two discrete functions and $(\theta_j),(\phi_j)$ the turning angles obtained as their respective partial sums. The following conditions are equivalent
\begin{enumerate}
\setcounter{enumi}{-1}
\item The polyline with curvature $(k_j)+\lambda(f_j)$ is closed $\forall \lambda \in \mathbb{R}$,
\item $\sum_{1 < j < N} \mathrm{e}^{i\theta_j} \phi_j^n = 0, \;\; \forall n \in \mathbb{N}_0$, 
\item $\sum_{j \in \phi^{-1}(a)} \mathrm{e}^{i\theta_j}=0, \;\; \forall a \in \mathbb{R}$.
\end{enumerate}
\end{theorem}
\begin{proof}
The equivalence ($0$) $\Leftrightarrow$ ($1$) is deduced as in \S \ref{first} by taking the $n$-th derivative with respect to $\lambda$ of the constant function $1=-\sum_{1 < j < N} \mathrm{e}^{i(\theta_j+\lambda \phi_j)}$. Condition ($1$) is easily implied by ($2$), while for the opposite direction we observe that for all $n\in\mathbb{N}, a\in \mathbb{R} \setminus \{0\}$,
$$
\sum_{1 < j < N} \mathrm{e}^{i\theta_j} \bigg(\frac{\phi_j}{a}\bigg)^n=\frac{1}{a^n} \sum_{1 < j < N} \mathrm{e}^{i\theta_j} \phi_j^n = 0.
$$
If all $\phi_j$'s are equal to $0$, we are done since the polyline associated to the turning angles $\theta_j$ is closed. Otherwise, letting $A=\max_j \{ \abs{\phi_j} \}$, 
\begin{equation*}
\begin{split}
0=\lim_{n \to \infty} \sum_{1 < j < N} \mathrm{e}^{i\theta_j} \bigg(\frac{\phi_j}{A}\bigg)^{2n} = \sum_{j \in \phi^{-1}(A)} \mathrm{e}^{i\theta_j}+\sum_{j \in \phi^{-1}(-A)} \mathrm{e}^{i\theta_j}, \\
0=\lim_{n \to \infty} \sum_{1 < j < N} \mathrm{e}^{i\theta_j} \bigg(\frac{\phi_j}{A}\bigg)^{2n+1} = \sum_{j \in \phi^{-1}(A)} \mathrm{e}^{i\theta_j}-\sum_{j \in \phi^{-1}(-A)} \mathrm{e}^{i\theta_j}, 
\end{split}
\end{equation*}
which entails $\sum_{j \in \phi^{-1}(A)} \mathrm{e}^{i\theta_j}=\sum_{j \in \phi^{-1}(-A)} \mathrm{e}^{i\theta_j}=0$.
For $A'=\max_j \{ \abs{\phi_j} \mid \abs{\phi_j}<A  \}$, it holds analogously
\begin{equation*}
\begin{split}
0=\lim_{n \to \infty} \sum_{1 < j < N} \mathrm{e}^{i\theta_j} \bigg(\frac{\phi_j}{A'}\bigg)^n & = \lim_{n \to \infty} \sum_{j \not \in \phi^{-1}(\pm A)} \mathrm{e}^{i\theta_j}  \bigg(\frac{\phi_j}{A'}\bigg)^n+ \bigg(\frac{\pm A}{A'}\bigg)^n  \sum_{j \in \phi^{-1}(\pm A)} \mathrm{e}^{i\theta_j} \\
& = \sum_{j \in \phi^{-1}(A')} \mathrm{e}^{i\theta_j} \pm  \sum_{j \in \phi^{-1}(-A')} \mathrm{e}^{i\theta_j} +0,
\end{split}
\end{equation*}
and we conclude by iterating the same argument until we exhaust all the finitely many vertices of the polyline.
\end{proof}
In order to find non-trivial pairs such that the polyline associated to $(k_j)+\lambda(f_j)$ is closed for all $\lambda$, by Theorem \ref{discretetheorem} the polyline associated to $(k_j)$ must possess at least one proper subset $\bar{V} \subset \{2,...,N-1\}$ of indices that is \textit{balanced}, meaning $\sum_{j \in \bar{V}} \mathrm{e}^{i\theta_j}=0    $. A visualization of this behaviour is given in Figure \ref{polyline}.

\begin{figure}
\begin{tikzpicture}
    \node[anchor=south west,inner sep=0] at (0,0) {\scalebox{1}{\includegraphics[width=\textwidth]{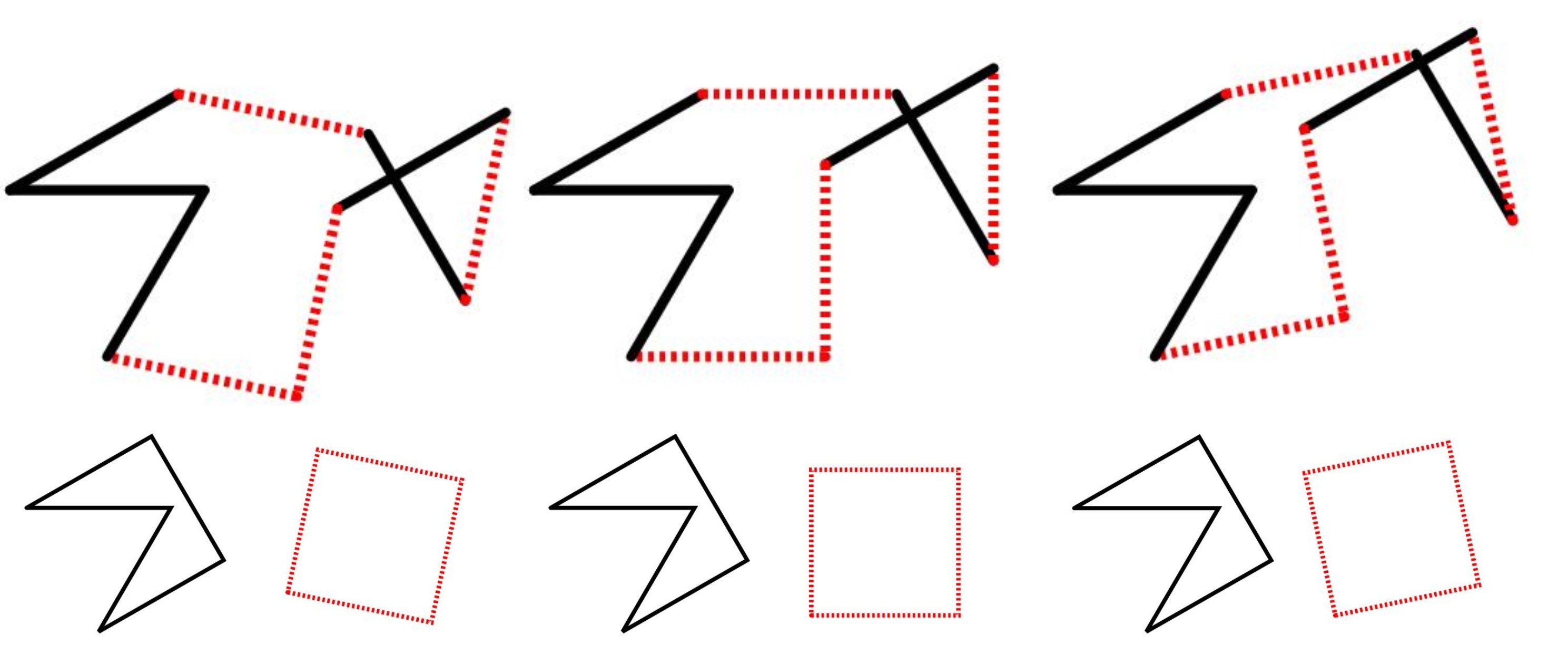}}};
     \node at (0.6,2.2) {$v_1,v_{10}$};
      \node at (1.95,3.7) {$v_2$};
      \node at (0.265,3.5) {$v_3$};
      \node at (1.5,4.75) {$v_4$};
      \node at (2.95,4.45) {$v_5$};
      \node at (3.8,2.61) {$v_6$};
      \node at (4.1,4.6) {$v_7$};
      \node at (3,3.4) {$v_8$};
      \node at (2.7,2.05) {$v_9$};
\end{tikzpicture}
\caption{The discrete curvature $(k_j)$ of a polyline is modified linearly in $\lambda$ to $(k_j)+\lambda(f_j)$, with $(f_j)= (0,0,\phi_1,-\phi_1,\phi_1,$ $-\phi_1, \phi_1,0)$. Such a curvature vector sums up to the turning angle $(0,0,\phi_1,0,\phi_1,0,\phi_1,\phi_1)$, which rotates, as $\lambda$ varies, only the dashed edges corresponding to a balanced subset of indices.} \label{polyline}
\end{figure}
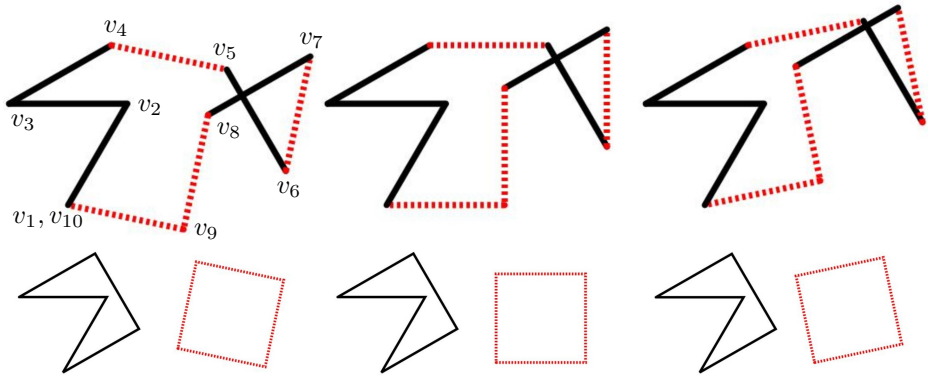

Note that it is easy to construct polylines with no balanced proper subsets of edges. Consider for example $n$ copies of the pair of unit vectors summing up to $(\frac{1}{n},0)$ and either the vector $(-1,0)$ for a polyline with an odd number of edges or the two unit vectors whose sum is $(-1,0)$ for an even number. Any proper subset of vectors from the ``copies'' part either consists of a single vector or its elements sum up to a non-unit vector different from $0$. In both case it is not possible to counterbalance the sum with the vector(s) on the other side of the y-axis.     

\section{There is no ``sufficient'' 4-vertex theorem.} \label{nonex}

The 4-vertex theorem provides a necessary condition for a function to be the curvature of a closed planar curve without self-intersections (see \cite{docarmo} or \cite{converse4} for a comprehensive survey). In this section we use the results from \S \ref{first} to show in a rigorous way that it is not possible to develop a sufficient condition of the same nature, namely it is not possible to tell whether a curvature function $k$ belongs to an arc-length parametrized closed curve by computing finitely many level sets and evaluations of $k$, its derivatives and its antiderivatives. The way we want to do this is by first assuming that in the class of analytic functions on $I$ such a procedure exists, and to show afterwards that it is always possible to construct an instance for which such a procedure yields the wrong answer.

We need to formalize what we mean with \textit{procedure}; this is done by first introducing the objects involved one by one, pointing out at the same time their high level meaning. We consider the set of sequences $A=\big\{\{a_{j}\}_{j \in \mathbb{N}}\big\}$, such that $a_{j} \in \mathbb{R} \cup \Sigma$, with $\Sigma$ a finite alphabet of symbols; sequences of this type will be used to store the progress of our procedure. Then we fix $L=\{l_j\}$, a countable set of analytic functions on $I$; this family will generalize the concept of level sets and finite linear combinations of its elements will be considered, selecting coefficients in $C=\big\{\{c_j\}_{j \in \mathbb{N}_0} \big\}$, set of real sequences whose terms are $0$ for $j$ big enough and $c_0 \in I$ (this is a special term used for evaluating functions at a certain parameter). 

Given the real analytic function $k$ that we want to test, our procedure is determined by three functions $M,H_k,G$.
\begin{align*}
M&: & \; A &\rightarrow  \{E,\cap\} \times \mathbb{Z} \times C \\
H_k&: & \;  \{E,\cap\} \times \mathbb{Z} \times C  &\rightarrow  \{E,\cap\} \times \mathbb{Z} \times A \\
G&: & \; A \times \{E,\cap\} \times \mathbb{Z} \times A  &\rightarrow \{\text{YES,NO}\}  \cup A,
\end{align*}
where $M$ and $G$ can be chosen arbitrarily and represent the functioning of the procedure while $H_k$ depends on the input $k$ and is the only tool we have to extract information about $k$ in the way we are going to specify in the next paragraph.

For $\{a_j\}\in A$, initialized to $a_j=\square$ with $\square \in \Sigma$ for all $j$, we compute iteratively $G\bigl(\{a_j\},H_k\big(M(\{a_j\}\big)\bigl)$ until an answer YES, the curve with curvature $k$ integrates in the arc-length to a closed curve, or NO, it does not, is output. Concerning the function $H_k$, for $(\sigma,n,\{c_j\}) \in \{E,\cap\} \times \mathbb{Z} \times C$, we have 
$$
 H_k(\sigma,n,\{c_j\}) = (\sigma,n,\{a_j\}),
$$
$$
 \text{with } \{a_j\} \textit{ storing } \begin{cases*}
      \text{the evaluation } k^{(n)}(c_0),   & if $\sigma=E$, \\
      \text{the solution(s) of } k^{(n)}=\sum_{j\geq 1} c_j l_j, & if $\sigma=\cap$,
    \end{cases*}
$$
where $k^{(n)}$ denotes the $n$-th derivative of $k$ for $n$ positive, $k$ itself for $n=0$ and the $n$-th antiderivative of $k$ for $n$ negative, meant as the result of taking $-n$ times the operation $\int_0^t$. With \textit{storing} we mean just sequentially writing the result of the evaluation or the solution(s) of the equation if they exist and are finitely many, or using special symbols from $\Sigma$ if there are no solutions or the two functions coincide (these are the only remaining possibilities in an analytic regime as ours). Unused terms of the sequence are just filled with the blank symbol $\square$. 

To summarize, $M$ looks at the current sequence $\{a_j\}$ and determines what is the informations $H_k$ should extract from $k$ or one of its antiderivatives/derivatives. Then, $G$ considers the result output from $H_k$ together with a copy of $\{a_j\}$ and either decides an answer to the problem or rather merges the new information updating the sequence in $A$. As anticipated, in the following we assume for a contradiction that there exist functions $M$, $G$ and a set $L$ such that the procedure $\mathfrak{T}$ they define in the sense above is \textit{correct}, meaning that, for any analytic functions $k$ on $I$, it decides an answer $\mathfrak{T}(k)\in $\{YES,NO$\}$ after finitely many iterations and that $\mathfrak{T}(k)=$YES if and only if the arc-length parametrized curve with curvature $k$ is closed.

\begin{figure}
\begin{tikzpicture}
\node[] at (0,0) {$k$};
\draw[very thick,->] (0.4,0) -- (1,0);
 \draw[very thick, ->] (4.75,1) arc (-45:225:1cm);
\node[align=left] at (4.1,1.7) {Finitely\\ many\\ times};
\node[draw,text width=5.2cm] at (4,0) {Compute an evaluation or a level set of either $k$, one of its derivatives or one of its antiderivatives.};
\draw[very thick,->] (7,0.1) -- (7.4,0.45);
\node[] at (7.9,0.45) {YES};
\draw[very thick,->] (7,-0.1) -- (7.4,-0.45);
\node[] at (7.9,-0.45) {NO};
\end{tikzpicture}
\caption{High level diagram of the decision procedure. Given a curvature function $k$, finitely many evaluations and/or level set of $k$, its derivatives and its antiderivatives are computed to output the answer YES, the associated curve is closed, or NO, it is not.} \label{diagproc}
\end{figure}
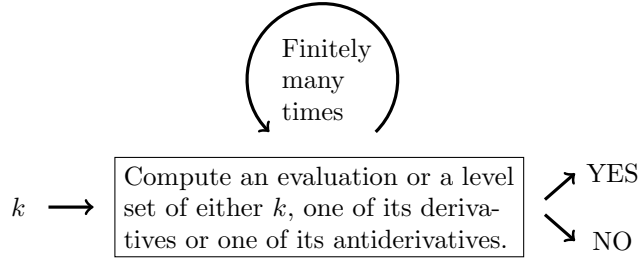

Our strategy is to perturbate the function $k$ of a closed curve to $k+\lambda p$, with $p$ another analytic function on $I$ and $\lambda$ a real number such that still $\mathfrak{T}(k+\lambda p)=$YES, but this time the associated curve is not closed anymore, entailing a contradiction. 
\begin{lemma} \label{lemma4vert}
If $\mathfrak{T}$ is a correct procedure, then there exist $k,P$ non-constant analytic functions on $I$ and $N\in\mathbb{N}$ such that $\mathfrak{T}\bigl(k+\lambda (\psi\cdot P)^{(N)}\bigl)=$YES for any $\psi$ analytic on $I$ and for all $\lambda \in \mathbb{R}$.
\end{lemma}
\begin{proof}
We start by considering the curvature function $k$ of a closed curve such that $k$, its derivatives and its antiderivatives are independent from the set $L$, \textit{i.e.} there is no finite combination of elements in $L$ that equals to any of them. Calling $\theta=\int k$, this can be done for example by observing that the condition $\int \mathrm{e}^{i\theta}=0$ for a curve to be closed allows a family with the cardinality of the continuum of independent functions or, more explicitely, by using the existence result from \S \ref{secexistence} and condition (2) of Theorem \ref{thmcond} to construct a closed curve with turning angle $\theta+g(\phi)$ with $g$ an appropriate analytic function that guarantees the independency from $L$.   

We construct then the set of triples $S=\{(t_j,d_j,n_j)\} \subset I\times \mathbb{N}\times\mathbb{Z}$, where the $t_j$'s are the single roots of the equations involving $k^{(n_j)}$ that the procedure solved to conclude the answer YES, and the $d_j$'s the respective degrees of such roots. At the same time, we put in $S$ also the triples $(t,1,n)$ if an evaluation of $k^{(n)}$ has been computed at $t$ over the run. Since $k$ has been chosen independent from $L$, the set $S$ is the complete record of what information has been extracted from $k$ by $H_k$. Calling $D:=\max_j \{d_j\}$ and $N:=\max_j \{|n_j|\}$, we define the polynomial
$$
P(t)=t^N \prod_j (t-t_j)^{2N+D}.
$$   
By construction, $k$ and $P$ as above satisfy, for all $\psi$ and all $\lambda$ small enough, $\mathfrak{T}\bigl(k+\lambda (\psi\cdot P)^{(N)}\bigl)=$YES; in fact, for small perturbations, $\mathfrak{T}$ performs exactly the same sequance of iterations and therefore ouputs the same result. In \S \ref{first} we saw that $F(\lambda)=\int_0^{2\pi} \mathrm{e}^{\bigl(i(k+\lambda (\psi\cdot P)^{(N)}\bigl)}$ is analytic in $\lambda$ and therefore, if the procedure is correct as we assumed, it actually holds $\mathfrak{T}\bigl(k+\lambda (\psi\cdot P)^{(N)}\bigl)=$YES for all $\lambda \in \mathbb{R}$.
\end{proof}

We are ready to prove the theorem promised at the beginning of the section.

\begin{theorem}
There is no sufficient 4-vertex theorem, \textit{i.e.} there is no correct procedure to determine with finitely many iterations whether the curve associated to the curvature function $k$ is closed by computing finitely many evaluations and/or generalized level sets of $k$, its derivatives and its antiderivatives.
\end{theorem}

\begin{proof}
Let $k$ and $P$ be chosen as in Lemma \ref{lemma4vert}. We construct $\psi$ for which it is apparent that condition (\ref{cond3}) cannot hold for the pair $\bigl(k,(\psi\cdot P)^{(N)}\bigl)$, therefore obtaining a contradiction. We choose $\bar{t}$ where $P(\bar{t})\neq0$  and consider the family of \textit{triangle functions} $T_{\delta}$, attaining the value $0$ outside the interval $[\bar{t}-\delta,\bar{t}+\delta]$ and linearly interpolating the value $T_{\delta}(\bar{t})=1/\delta$.
%$$
%T_{\delta}:=\frac{1}{\delta^2}
%\begin{cases*}
%0, & $t<\bar{t}-\delta$ \\
%t-(\bar{t}-\delta), & $\bar{t}-\delta \leq t < \bar{t}$ \\
%-t+(\bar{t}+\delta), & $\bar{t} \leq t < \bar{t}+\delta$ \\
%0, & $\bar{t}+\delta \leq t$.
%\end{cases*}
%$$
By the Stone-Weierstrass theorem, for $\varepsilon>0$, we can find a polynomial $h_{\varepsilon,\delta}$ such that 
$$
\sup_{I} |h_{\varepsilon,\delta}-T_{\delta}|< \varepsilon \text{ and hence } \sup_{I} \bigl|h_{\varepsilon,\delta}^{(-j)}-T_{\delta}^{(-j)}\bigl|< \varepsilon(2\pi)^{j}, \; \forall j\in\mathbb{N}, 
$$
where the superscript $(-j)$ means the $j$-th antiderivative $\int_0^t$ of a function. For $j \geq 1$, it holds 
$$
\sup_{I} h_{\varepsilon,\delta}^{(-j)} < (2\pi)^{j-1} + \varepsilon(2\pi)^{j}.
$$

We consider now
\begin{equation} \label{derivatives}
\bigl(h_{\varepsilon,\delta}^{(-N)} \cdot P\bigl)^{(N)}=h_{\varepsilon,\delta}P+\sum_{1\leq j\leq N} \binom{N}{j} h_{\varepsilon,\delta}^{(-j)} P ^{(j)}.
\end{equation}
With $M=\max_{1\leq j \leq N}\sup_{I} P^{(j)}$, we see that the second term of (\ref{derivatives}) is bounded by $\sum_{1\leq j\leq N} \binom{N}{j} (2\pi)^{j}(1+\varepsilon)M$, which does not depend on $\delta$. Choosing $\psi=h_{\varepsilon,\delta}$ with $\varepsilon$ and $\delta$ small enough, the maxima on $I$ of $\bigl(h_{\varepsilon,\delta}^{(-N)} \cdot P\bigl)^{(N)}$ are all contained in an arbitrarily small neighborhod of $\bar{t}$. This makes it  impossible to satisfy condition (\ref{cond3}), which is the contradiction we needed to conclude the theorem.

\end{proof}

\begin{remark} Showing the impossibility of a sufficient analogue of the 4-vertex theorem cannot be reduced to a cardinality argument. For example, if we are just interested in constructing a procedure as the one described at the beginning of the section for curvature functions over $I$ that are $\pi$-periodic and such that $|k^{(-1)}(s)| \leq \pi$ for $s\leq \pi$, then the associated curve is closed if and only if $k^{(-1)}(\pi)=\pi$ and it is therefore enough to compute such an evaluation to conclude the correct answer. The interplay between closed curves and periodic curvature functions has been characterized in \cite{MR2408486}.
\end{remark}

\subsection*{Future work.} Given the curvature function $k$ of a closed curve, when is it possible to find $f$ such that the curve associated to $k+\lambda f$ is closed for all $\lambda$? In \S \ref{secexistence} we identified a class of pairs of functions that satisfies this condition and the next obvious step would be a full characterization in the $C^h$ and analytic setting. Thinking in terms of the turning angle $\theta=\int k$, a possible way of approaching the problem could be by synthesizing a Fourier series  for $\phi$ that would satisfy the family of orthogonality relations $\int \mathrm{e}^{i\theta}\phi^n=0$ in $L^2$. The ugliness of the convolution formula for the Fourier coefficients of a product prevented the author from succeeding. 

Another nice improvement would be the generalization of the periodicity result from \S \ref{per} to the case $\phi(0)$ being a critical value. This would also make the proof of the non-existence of a ``sufficient'' 4-vertex theorem in \S \ref{nonex} more agile.

\section*{Acknowledgments.}
The author acknowledges the support of the Austrian Science Fund (FWF): W1230, ``Doctoral Program Discrete Mathematics'' and of SFB-Transregio 109 ``Discretization in Geometry \& Dynamics'' funded by DFG and FWF (I2978).

The author would also like to thank Johannes Wallner and Felix Dellinger for helpful discussions and the community of \textit{MathOverflow} for useful hints.

\bibliography{bibliography} 
\bibliographystyle{amsplain}

\end{document}